\documentclass[12pt,english]{article}
\usepackage[T1]{fontenc}
\usepackage[latin9]{inputenc}
\usepackage{authblk}
\usepackage{breakcites} %for breaking the citations
\usepackage{geometry}
\usepackage{amsmath}
\usepackage{amsthm}
\usepackage{amssymb}
\usepackage{comment}

\makeatletter

\numberwithin{equation}{section}
\theoremstyle{plain}
\newtheorem{thm}{\protect\theoremname}
\theoremstyle{plain}
\newtheorem{lem}{\protect\lemmaname}
\theoremstyle{plain}
\newtheorem{cor}{\protect\corollaryname}

\usepackage{color}

\usepackage{graphicx}
\usepackage{float}

\date{} 

\makeatother

\usepackage{babel}
\providecommand{\corollaryname}{Corollary}
\providecommand{\lemmaname}{Lemma}
\providecommand{\theoremname}{Theorem}

\begin{document}
\title{When Is the Conway-Maxwell-Poisson Distribution Infinitely Divisible?}

\author[1]{Xi Geng\thanks{{\sf{email: xi.geng@unimelb.edu.au}}.}}
\author[1]{Aihua Xia\thanks{{\sf{email: aihuaxia@unimelb.edu.au}}. Work supported by the Australian Research Council Grant No  DP190100613.}}

%\affil[1]{%
\affil[1]{%
School of Mathematics and Statistics, the University of Melbourne, Parkville VIC 3010, Australia} 

\maketitle
\vskip-1cm
\begin{abstract}
	An essential character for a distribution to play a central role in the limit theory is infinite divisibility. In this note, we prove that the Conway-Maxwell-Poisson (CMP) distribution is infinitely divisible iff it is the Poisson or geometric distribution. This explains that, despite its applications in a wide range of fields, there is no theoretical foundation for the CMP distribution to be a natural candidate for the law of small numbers. 
\end{abstract}

\vskip8pt \noindent\textit{Key words and phrases:}  Conway-Maxwell-Poisson distribution, infinite divisibility, entire function.

\vskip8pt\noindent\textit{AMS 2020 Subject Classification:}
primary 60F05; secondary 60E05, 60E07. %checked by Aihua on 7 Nov 2020.

\section{Introduction and the main result}

The fundamental driving force for the success of modelling count data is the law of small numbers and the Poisson distribution is  undoubtedly the cornerstone of the approximation theory in capturing the distribution of counts of rare events \cite{BHJ}. The main disadvantage of the Poisson distribution is that its mean and variance are identical, so when the count data is over-dispersed (resp. under-dispersed), i.e. its variance is bigger (resp. less) than its mean, the Poisson distribution is inadequate in fitting the data. There are many alternatives aiming to overcome the deficiency, e.g., compound Poisson (including negative binomial), translated Poisson \cite{BX99,Roellin05} and convolutions of infinitely divisible distributions \cite{Presman83,Kruopis86}. By introducing an extra parameter to the Poisson distribution for looking after the dispersion behaviour of the count data, the \textit{Conway-Maxwell-Poisson distribution} ${\rm CMP}(\lambda,\nu)$ \cite{Conway62} assumes the probability
mass function
\[
\mathbb{P}(X=k)=\frac{1}{Z(\lambda,\nu)}\cdot\frac{\lambda^{k}}{(k!)^{\nu}},\ k=0,1,2,\cdots,
\]
where the parameters $(\lambda,\nu)$ satisfy $\lambda,\nu>0$ or
$\nu=0,\lambda\in(0,1)$, and $Z(\lambda,\nu)$ denotes the normalising
constant. The CMP distribution is over-dispersed if $\nu<1$ and under-dispersed if $\nu>1$. Due to its smooth transition between over-dispersion and under-dispersion, it
plays a significant role in modelling count data \cite{Hilbe14} and has an extraordinarily diverse range of applications, see \cite{Sellers12} for a brief survey. However, despite some initiatives \cite{Daly16,Li2020}, there is disproportionately little advance in its analytical properties and an approximation theory based on this popular distribution. This note aims to explain the fundamental reason behind the unbalanced development.

An essential character for a distribution to play a central role in the limit theory is infinite divisibility. For the CMP distribution, we have the following result.

\begin{thm}
\label{thm:main}The distribution ${\rm CMP}(\lambda,\nu)$ is infinitely
divisible if and only if $\nu=0$ or $\nu=1$.
\end{thm}

In other words, the CMP distribution is infinitely divisible iff it is the Poisson or geometric distribution. This explains that, despite its applications in a wide range of fields, the only member in the CMP family that can capture the distribution of the count of weakly dependent rare events is the Poisson distribution and there is no theoretical foundation for the CMP distribution to be a natural candidate for the law of small numbers. In particular, Theorem~\ref{thm:main} implies that the CMP process possessing independent stationary increments with $\nu>0$ in \cite{Zhu17} does not exist except it is a Poisson process.

\section{The proof}

The sufficiency part requires no comment. In addition, the distribution on nonnegative integers is infinitely divisible if and only if it is a compound Poisson \cite[p.~290]{Feller68}, and as such compound Poisson cannot be under-dispersed (see subsection~\ref{decomp1}), 
we know that ${\rm CMP}(\lambda,\nu)$ is not infinitely divisible
when $\nu>1.$ By comparing with the probability mass function of a compound Poisson distribution, \cite{Mao20} concludes that ${\rm CMP}(\lambda,\nu)$ is not infinitely divisible when $\nu\in (0.33,1)$. However, as $\nu$ edges towards $0$, the elementary approach of comparing the probability mass functions offers little hope in determining infinite divisibility of the CMP distribution, we need an entirely different approach to tackle the problem. Our strategy is to use complex analysis.

Throughout the rest, let $\nu\in(0,1)$ and $\lambda>0$ be given. We
assume on the contrary that $X\sim{\rm CMP}(\lambda,\nu)$ is infinitely
divisible. The main steps towards reaching a contradiction are
summarised as follows.
\begin{enumerate}
\item Under the assumption, $X$ must be a compound Poisson random variable,
say 
\[
X=\sum_{n=1}^{N}Y_{n},
\]
where $N$ is Poisson distributed and $\{Y_{1},Y_{2},\cdots\}$ are
i.i.d. positive integer-valued random variables and independent of $N$.
\item The probability generating function $G(z)$ of $Y_{1}$ must be an
entire function, i.e. holomorphic on the entire complex plane.
\item The exponential-type growth property of the probability generating
function $F(z)$ of $X$ forces $G(z)$ to have no more than polynomial growth.
\item $G(z)$ must be a polynomial as a consequence of complex analysis.
\item The precise growth estimate for $F(z)$ obtained in Step 3 further
forces $G(z)$ to be a monomial, which then leads to a contradiction
trivially.
\end{enumerate}
In the rest of this note, we develop the above steps carefully.

\subsection{The compound Poisson decomposition and related probability generating
functions}\label{decomp1}

Under the assumption of infinite divisibility, we know from \cite[p.~290]{Feller68}
 that $X$ is compound Poisson, i.e. 
\[
X\stackrel{d}{=}\sum_{n=1}^{N}Y_{n},
\]
where $N\sim{\rm Pn}(\mu)$ with some $\mu>0,$ $\{Y_{1},Y_{2},\cdots\}$
is an i.i.d. sequence of {$\mathbb{N}\triangleq\{0,1,2,\dots\}$}-valued random variables that
is independent of $N$. By absorbing the mass of $Y_{1}$ at the origin
to the parameter of $N$ if necessary, we may assume that $Y_{1}$
takes values in the positive integers {$\mathbb{N}_{+}\triangleq\{1,2,\dots\}$}. 
{Note that in this case we have
\[
\mathbb{V}[X] - \mathbb{E}[X] = \mathbb{E}[N]\big(\mathbb{E}[Y_1^2]-\mathbb{E}[Y_1]\big)\geqslant 0. 
\]This already implies that $\nu$ cannot be greater than $1$ if $\mathrm{CMP}(\lambda,\nu)$ were infinitely divisible.}

We now write
\[
q_{k}\triangleq\mathbb{P}(Y_{1}=k),\ \ \ k\in{\mathbb{N}_{+}},
\]
and let $G(z)$ denote the probability generating function of $Y_{1}$,
i.e. 
\begin{equation}
G(z)=\sum_{r=1}^{\infty}q_{r}z^{r}.\label{eq:G(z)}
\end{equation}
We treat $z$ as a complex variable and note that $G(z)$ is holomorphic
at least in the unit disk. From elementary probability theory, $e^{\mu(G(z)-1)}$
is the probability generating function of $X.$ By using the ${\rm CMP}$-distribution
of $X,$ we have 
\begin{equation}
e^{\mu G(z)}=\frac{e^{\mu}}{Z(\lambda,\nu)}\sum_{k=0}^{\infty}\frac{(\lambda z)^{k}}{(k!)^{\nu}}\label{eq:PGF}
\end{equation}
for those $z$'s within the radius of convergence of $G(z)$. By taking
$z=0,$ we get $e^{\mu}=Z(\lambda,\nu)$ and thus 
\[
e^{\mu G(z)}=\sum_{k=0}^{\infty}\frac{(\lambda z)^{k}}{(k!)^{\nu}}.
\]
Let us set 
\[
F(z)\triangleq\sum_{k=0}^{\infty}\frac{(\lambda z)^{k}}{(k!)^{\nu}}.
\]
Note that $F(z)$ defines an entire function.

\subsection{Holomorphicity of $G(z)$ on $\mathbb{C}$}

Our next step is to show that $G(z)$ must be an entire function.
Before doing so, we first recall the construction of the logarithm
of a holomorphic function (cf. \cite[p. 123]{Lang99}). Suppose
that $f(z)$ is a holomorphic function on a simply connected domain
$\Omega$ and is everywhere non-zero, then there exists a holomorphic
function $L(z)$ on $\Omega$ such that 
\begin{equation}
e^{L(z)}=f(z).\label{eq:LogF}
\end{equation}
Indeed, the equation (\ref{eq:LogF}) suggests that $L'(z)=\frac{f'(z)}{f(z)}$,
which leads us to defining
\[
L(z)\triangleq\int_{z_{0}}^{z}\frac{f'(w)}{f(w)}dw,\ \ \ z\in\Omega,
\]
where $z_{0}$ is a fixed based point in $\Omega$ and the integral
is performed along an arbitrary path joining $z_{0}$ to $z$. The
well-definedness of $L(z)$ is a simple consequence of the simply
connectedness of $\Omega$ and Cauchy's theorem.
\begin{lem}
\label{lem:S1}The function $G(z)$ defined by the power series (\ref{eq:G(z)})
is an entire function.
\end{lem}
\begin{proof}
Let $R$ be the radius of convergence of $G(z)$. Suppose on the contrary
that $R<\infty.$ We know that 
\[
e^{\mu G(z)}=F(z),\ \ \  z\in B_{R}\triangleq\{z:|z|<R\}.
\]
In particular, 
\[
e^{\mu G(\rho)}=F(\rho),\ \ \ \rho\in(0,R).
\]
Since the coefficients of $G(z)$ are non-negative, by the monotone
convergence theorem we have 
\[
e^{\mu G(R)}=F(R)<\infty.
\]
It follows that $G(z)$ is convergent on the entire boundary of $B_{R}.$
The dominated convergence theorem further implies that $G(z)$ is
continuous on the closed ball $\overline{B_{R}}$ and we thus have
\begin{equation}
e^{\mu G(z)}=F(z),\ \ \ z\in\overline{B_{R}}.\label{eq:ExpClosedBall}
\end{equation}

Now let $U\triangleq\{z\in\mathbb{C}:F(z)\neq0\}.$ Then $U$ is an
open subset of $\mathbb{C}.$ Since the exponential function is everywhere
non-vanishing, it follows from (\ref{eq:ExpClosedBall}) that $\overline{B_{R}}\subseteq U,$
and hence $B_{R+\varepsilon}\subseteq U$ for some $\varepsilon>0.$
As $B_{R+\varepsilon}$ is simply connected, there is a well-defined
logarithm of $F(z)$ on $B_{R+\varepsilon},$ namley a holomorphic
function $L(z)$ such that 
\[
e^{L(z)}=F(z),\ \ \ z\in B_{R+\varepsilon}.
\]
Combining with (\ref{eq:ExpClosedBall}), we obtain
\[
L(z)=\mu G(z)+2\pi ik(z),\ \ \ z\in B_{R}
\]
with some function $k:B_{R}\rightarrow\mathbb{Z}\triangleq\{0,\pm1,\pm2,\dots\}$. Since both $L(z)$
and $G(z)$ are continuous, the function $k(z)$ must be constant
(say $k(z)\equiv k^{*}$) and we arrive at 
\[
L(z)=\mu G(z)+2\pi k^{*}i=2\pi k^{*}i+\sum_{r=1}^{\infty}\mu q_{r}z^{r}.
\]
The power series on the right hand side gives the Taylor expansion
of $L(z)$. Since $L(z)$ is holomorphic on $B_{R+\varepsilon},$
its radius of convergence must be at least $R+\varepsilon.$ This contradicts
the assumption that $R$ is the radius of convergence for $G(z)$.
Therefore, the series (\ref{eq:G(z)}) is convergent on the entire
complex plane and $G(z)$ is thus an entire function.
\end{proof}

\subsection{Precise growth-type estimate of $F(z)$}

In this part, we investigate the precise growth of $F(z)$. This is
the core step of the argument. 

We first prepare a simple analytical lemma. 
\begin{lem}
\label{lem:StirComp}There exist positive constants $C_{1},C_{2}$
and $K$ depending only on $\nu$, such that 
\begin{equation}
C_{1}k^{\frac{{1-\nu}}{2}}\nu{}^{\nu k}\leqslant\frac{(k\nu)!}{(k!)^{\nu}}\leqslant C_{2}k^{\frac{{1-\nu}}{2}}\nu^{\nu k}\label{eq:StirComp}
\end{equation}
for all $k\geqslant K.$
\end{lem}
\begin{proof}
We recall the following Stirling's approximation for the Gamma function for all positive real $x$:
\begin{equation}
x!\triangleq\Gamma(x+1)\sim\sqrt{2\pi x}\left(\frac{x}{e}\right)^{x}\ \ \ \text{as }x\rightarrow\infty.
\end{equation}
As a result, we have
\[
\frac{(k\nu)!}{(k!)^{\nu}}\sim\sqrt{\nu}(2\pi k)^{\frac{1-\nu}{2}}\nu^{\nu k}\ \ \ \text{as }k\rightarrow\infty,
\] and the claim thus follows. 
\end{proof}

The main result for this part is stated below. It quantifies the
precise growth rate of $F(z)$ as $z$ approaches infinity along the positive axis.
\begin{lem}
\label{lem:S3}Let $M\triangleq\lambda\nu{}^{\nu}.$ There exists a
constant $C>0$ as well as two polynomials $p_{1},p_{2}$ with positive
coefficients such that 
\begin{equation}
\frac{C}{R^{{1/\nu}}}e^{(MR)^{1/\nu}}-p_{1}(R)\leqslant F(R)\leqslant p_{2}(R)e^{(MR)^{1/\nu}}\label{eq:GrowF}
\end{equation}
for all $R$ with $MR>1$.
\end{lem}
\begin{proof}
We first establish the upper bound. Let $K$ be as in Lemma \ref{lem:StirComp}. Enlarging $K$ if necessary, we assume $K\nu\ge 2$. For $k\ge K$, the right hand side of \eqref{eq:StirComp} ensures
$$\frac{(\lambda R)^k}{(k!)^\nu}\le C_2\frac{(MR)^k}{(k\nu)!}k^{\frac{1-\nu}{2}}\le C_3\frac{(MR)^k}{(k\nu-1)!},$$
where $C_3\triangleq C_2/\nu.$ Hence,
\begin{equation}
F(R)=\sum_{k=0}^{K-1}\frac{(\lambda R)^{k}}{(k!)^{\nu}}+\sum_{k=K}^{\infty}\frac{(\lambda R)^{k}}{(k!)^{\nu}}\leqslant\sum_{k=0}^{K-1}\frac{(\lambda R)^{k}}{(k!)^{\nu}}+C_3\sum_{k=K}^{\infty}\frac{(MR)^{k}}{(k\nu-1)!}.\label{lemma3-1}
\end{equation}
We now examine the last summation in the above inequality by introducing
the division $k=mp+r$ where $p\triangleq1/\nu$, $m\in{\mathbb{N}}$
and $r\in[0,p).$ To be more precise, for each $m\in{\mathbb{N}}$ we
set
\[
R_{m}\triangleq\{r\in[0,p):mp+r\in\mathbb{Z}\}.
\]
Note that $R_{m}\neq\emptyset$ (since $p>1$) and contains at most
$[p]+1$ elements, {where $[p]$ denotes the integer part of $p$}. It is clear that each $k\in{\mathbb{N}}$ can be
written as $k=mp+r$ with some $m\in{\mathbb{N}}$ and $r\in R_{m}$.
To see the uniqueness of such decomposition, suppose that 
\[
mp+r=m'p+r'.
\]
Then 
\[
(m-m')p=r'-r\in(-p,p).
\]
As a result, $m=m'$ and $r=r'$. It follows that $k\leftrightarrow(m,r)$
is a one-to-one correspondence. By using this decomposition, we have
\begin{align*}
\sum_{k=K}^{\infty}\frac{(MR)^{k}}{(k\nu-1)!} & \le \sum_{m=1}^{\infty}\sum_{r\in R_{m}}\frac{(MR)^{mp+r}}{(m+r\nu-1)!}\\
 & \leqslant\sum_{m=1}^{\infty}\frac{(MR)^{mp}}{(m-1)!}\sum_{r\in R_{m}}(MR)^{r}\\
 & \leqslant([p]+1)(MR)^{2p}e^{(MR)^{p}},
\end{align*}provided that $MR>1$. Since the second last summation in
(\ref{lemma3-1}) is polynomial, the desired upper bound follows.

The idea of establishing the lower bound is similar. By using Lemma
\ref{lem:StirComp}, we have 
\begin{align*}
F(R) & \geqslant C_{1}\sum_{k=K}^{\infty}\frac{k^{\frac{1-\nu}{2}}(MR)^{k}}{(k\nu)!}\geqslant C_{1}\sum_{k=K}^{\infty}\frac{(MR)^{k}}{(k\nu)!}\\
 & =C_{1}\left(\sum_{k=0}^{\infty}\frac{(MR)^{k}}{(k\nu)!}-\sum_{k=0}^{K-1}\frac{(MR)^{k}}{(k\nu)!}\right)\\
 & =C_{1}\sum_{m=0}^{\infty}\sum_{r\in R_{m}}\frac{(MR)^{mp+r}}{(m+r\nu)!}-p_{1}(R),
\end{align*}
where 
\[
p_{1}(R)\triangleq C_{1}\sum_{k=0}^{K-1}\frac{(MR)^{k}}{(k\nu)!}.
\]
To estimate the double summation, we observe that (assuming $MR>1$)
\[
\sum_{r\in R_{m}}\frac{(MR)^{mp+r}}{(m+r\nu)!}\geqslant\frac{(MR)^{mp}}{(m+1)!}\sum_{r\in R_{m}}(MR)^{r}\geqslant\frac{(MR)^{mp}}{(m+1)!}
\]
where the last inequality follows from the fact that $R_{m}\neq\emptyset$.
Therefore,
\[
\sum_{m=0}^{\infty}\sum_{r\in R_{m}}\frac{(MR)^{mp+r}}{(m+r\nu)!}\geqslant\sum_{m=0}^{\infty}\frac{(MR)^{mp}}{(m+1)!}=(MR)^{-p}\big(e^{(MR)^{p}}-1\big).
\]
The desired lower bound thus follows.

\begin{comment}
The idea of establishing the lower estimate is similar. By using Lemma
\ref{lem:StirComp}, we have 
\begin{align}
F(R) & \geqslant C_{1}\sum_{k=K}^{\infty}\frac{k^{\frac{\nu-1}{2}}(MR)^{k}}{(k\nu)!}\geqslant \frac{C_{3}}{R}\sum_{k=K-1}^{\infty}\frac{(MR)^{k}}{(k\nu)!}\nonumber\\
% & =C_{1}\big(\sum_{k=0}^{\infty}\frac{(MR)^{k}}{(k\nu)!}-\sum_{k=0}^{K-1}\frac{(MR)^{k}}{(k\nu)!}\big)\\
 & =\frac{C_{3}}{R}\sum_{m=0}^{\infty}\sum_{r\in R_{m}}\frac{(MR)^{mp+r}}{(m+r\nu)!}-p_{1}(R),\label{lemma3-2}
\end{align}
where 
\[
p_{1}(R)\triangleq \frac{C_{3}}{R}\sum_{k=0}^{K-2}\frac{(MR)^{k}}{(k\nu)!}.
\]
To estimate the double summation in (\ref{lemma3-2}), we observe that (assuming $MR>1$)
\[
\sum_{r\in R_{m}}\frac{(MR)^{mp+r}}{(m+r\nu)!}\geqslant\frac{(MR)^{mp}}{(m+1)!}\sum_{r\in R_{m}}(MR)^{r}\geqslant\frac{(MR)^{mp}}{(m+1)!},
\]
where the last inequality follows from the fact that $R_{m}\neq\emptyset$.
Therefore,
\[
\sum_{m=0}^{\infty}\sum_{r\in R_{m}}\frac{(MR)^{mp+r}}{(m+r\nu)!}\geqslant\sum_{m=0}^{\infty}\frac{(MR)^{mp}}{(m+1)!}=(MR)^{-p}\big(e^{(MR)^{p}}-1\big).
\]
The desired lower estimate thus follows.
\end{comment}
\end{proof}

\subsection{The function $G(z)$ is a polynomial}

Let $d\geqslant2$ be the unique integer such that $1/\nu\in(d-1,d].$
It follows from Lemma \ref{lem:S3} that 
\[
F(R)\leqslant e^{M'R^{d}}\ \ \ \mbox{for } R\ \text{sufficiently large}
\]
with some $M'>0$. On the other hand, since $G(z)$ has non-negative
coefficients, we have
\[
e^{\mu|G(z)|}\leqslant e^{\mu G(|z|)}=F(|z|).
\]
As a result, we have
\[
|G(z)|\leqslant\frac{1}{\mu}\log F(|z|)\leqslant\frac{M'}{\mu}|z|^{d}\ \ \ \mbox{for } z\ \text{large}.
\]
Since $G(z)$ is an entire function, the following complex analysis
lemma implies that $G(z)$ has to be a polynomial.
\begin{lem}
Let $g(z)$ be an entire function such that 
\[
|g(z)|\leqslant C|z|^{d}
\]
for all large $z$. Then $g(z)$ is a polynomial of degree at most
$d$.
\end{lem}
\begin{proof}
We write 
\[
g(z)=\sum_{k=0}^{\infty}a_{k}z^{k},\ \ \ z\in\mathbb{C}.
\]
From Cauchy's integral formula, the Taylor coefficients are given
by 
\[
a_{k}=\frac{1}{2\pi i}\int_{\partial B_{R}}\frac{g(z)}{z^{k+1}}dz=\frac{1}{2\pi}\int_{0}^{2\pi}\frac{g(Re^{i\theta})}{(Re^{i\theta})^{k}}d\theta,\ \ \ k=0,1,2,\cdots.
\]
Note that the above formula is true for all $R>0.$ By the assumption,
we have
\[
|a_{k}|\leqslant\frac{C}{2\pi}\int_{0}^{2\pi}\frac{R^{d}}{R^{k}}d\theta=CR^{d-k}
\]
for all large $R$. For each $k>d,$ by taking $R\rightarrow\infty$
we conclude that $a_{k}=0$.
\end{proof}
\begin{cor}
\label{cor:S4}The function $G(z)$ is a polynomial of degree at most
$d$, i.e. $G(z)=q_{1}z+\cdots+q_{d}z^{d}$. In particular, $Y_{1}$
is supported on $\{1,\cdots,d\}$.
\end{cor}

\subsection{Reaching the contradiction}

We are now in a position to complete the proof of Theorem \ref{thm:main}.
The main point is that the growth estimate (\ref{eq:GrowF}) of $F(R)$
forces the polynomial $G(z)$ to consist of the single monomial $q_{d}z^{d}$
only, which then trivially leads to a contradiction since $X$ achieves
all possible values in $\mathbb{N}$. Recall that $d\geqslant2$ is
the unique integer such that $1/\nu\in(d-1,d]$.

\begin{proof}[Proof of Theorem \ref{thm:main}]

According to Lemma \ref{lem:S3} and Corollary \ref{cor:S4}, we have
\begin{equation}
\frac{C}{R^{1/\nu}}e^{(MR)^{1/\nu}}-p_{1}(R)\leqslant e^{\mu(q_{1}R+\cdots+q_{d}R^{d})}\leqslant p_{2}(R)e^{(MR)^{1/\nu}}\label{eq:UpperS5}
\end{equation}
for all sufficiently large $R.$ The lower bound forces $q_{d}$
to be non-zero. In the case when $\nu\neq1/d$, the upper bound
in (\ref{eq:UpperS5}) cannot hold true when $R\rightarrow\infty,$
giving a contradiction. It now remains to consider the
case when $\nu=1/d$. In this case, the upper bound implies $\mu q_{d}\leqslant M^{1/\nu}$
while the lower bound implies that $\mu q_{d}\geqslant M^{1/\nu}.$
Therefore, $\mu q_{d}=M^{1/\nu}.$ If $q_{1},\cdots,q_{d-1}$ were
not all zero, the upper bound cannot hold true. As a result, we
have $q_{1}=\cdots=q_{d-1}=0$. In particular, $q_{d}=1$ and $G(z)=z^{d}.$
Since $d\geqslant2,$ this clearly contradicts the fact that $\mathbb{P}(X=1)>0.$

Now the proof of Theorem \ref{thm:main} is complete.

\end{proof}

\end{document}